\documentclass{amsart}

\usepackage[T1]{fontenc}
\usepackage[utf8]{inputenc}
\usepackage[english]{babel}
\usepackage{amsmath,amssymb,amstext}
\usepackage{amsthm,amsfonts}
\usepackage{enumitem}
\usepackage[normalem]{ulem}
\usepackage{mathtools}
\usepackage{xcolor}
\usepackage{hyperref}
\usepackage{tikz}
\usepackage{xurl}

\newcommand{\abs}[1]{\left| #1 \right|}

\newcommand{\NN}{\mathbb{N}}
\newcommand{\QQ}{\mathbb{Q}}
\newcommand{\RR}{\mathbb{R}}
\newcommand{\CC}{\mathbb{C}}
\newcommand{\ZZ}{\mathbb{Z}}

\DeclareMathOperator{\Gal}{Gal}

\theoremstyle{plain}
\newtheorem{theorem}{Theorem}[section]
\newtheorem{prop}[theorem]{Proposition}
\newtheorem{lemma}[theorem]{Lemma}
\newtheorem{cor}[theorem]{Corollary}

\theoremstyle{definition}

\theoremstyle{remark}

\title{Complete Resolution Of A Family Of Twisted Thue Equations}

\author[T. Hilgart]{Tobias Hilgart}
\address{T. Hilgart,
University of Salzburg,
Hellbrunnerstrasse 34/I,
A-5020 Salzburg, Austria}
\email{tobias.hilgart\char'100proton.me}

\author[C. Premstaller]{Carina Premstaller}
\address{C. Premstaller,
University of Salzburg,
Hellbrunnerstrasse 34/I,
A-5020 Salzburg, Austria}
\email{carina.premstaller\char'100plus.ac.at}

\author[V. Ziegler]{Volker Ziegler}
\address{V. Ziegler,
University of Salzburg,
Hellbrunnerstrasse 34/I,
A-5020 Salzburg, Austria}
\email{volker.ziegler\char'100plus.ac.at}

\begin{document}

\begin{abstract}
 One of the first infinite families of Thue equations,
 $$F_n(X)=X^3 - (n-1) X^2Y - (n+2)XY^2 - Y^3 = \pm 1$$ for $n\in \ZZ$,
 was solved by Thomas in 1990.
 This family is associated to the simplest cubic fields $\QQ(\lambda)$ of Shanks,
 where $\lambda$ is a root of $F_n(X,1)$.
 Levesque and Waldschmidt twisted the Thue equations by an exponent $t$ and 
 looked at the equation
 $$N_{\QQ(\lambda)/\QQ}(X-\lambda^t Y)=\pm 1,$$
 where $t\in \ZZ$ with $t\neq 0$. 
 In this paper, we find all solutions $(X,Y,n,t)\in \ZZ^4$ with 
 $n,t\in\ZZ$ and $t\neq 0$ to this family of twisted Thue equations, 
 thereby answering a question of Levesque and Waldschmidt.
\end{abstract}

\maketitle
\begin{center}
\textit{Dedicated to Michel Waldschmidt on the occasion of his 80th birthday.}
\end{center}
\section{Introduction}
	
	For a given Thue equation, all integer solutions can be explicitly determined. 
    In particular, Baker’s work \cite{baker0} provides explicit upper bounds
	for the size of solutions, reducing the problem to a finite computation. 
    Combined with reduction techniques and modern computational tools, 
    this yields a practical method for determining all integer solutions of a 
    given Thue equation (see e.g. \cite{Tzanakis:1989}).
	
	In 1990, Thomas \cite{thomas} considered the family of polynomials associated 
    to Shanks' simplest cubic fields,
	\begin{equation}
		\label{eq-shanks}
		F_n(X,Y) = X^3 - (n-1) X^2Y - (n+2)XY^2 - Y^3 = \pm 1. 
	\end{equation}
	He effectively proved that for $n > 1.3675 \cdot 10^7$, there are no 
	non-trivial solutions, i.e. no triple $(n,x,y)\in \ZZ$ such that
	$$n \geq 0, \quad \max\{|x|,|y|\} \geq 2 \quad \text{and}
		\quad F_n(x,y) = \pm 1 $$
    exists.        
	Mignotte \cite{mignotte} completed Thomas' work in 1993,
	by solving the problem for all $n$.

	A complete list of solutions is given by
    \begin{center}
        $\begin{array}{c|cccccc}
             n & & & (x,y) & & & \\
             \hline
             0 & (5,4) & (-4,9) & (-9,5) & (-2,1) & (1,1) & (-1,2) \\
             1 & (2,1) & (-3,2) & (1,3) \\
             3 & (7,2) & (-2,9) & (-9,7) & & &
        \end{array}$
    \end{center}
	
	Following Thomas' notation, we sort the roots of $f_n(X) = F_n(X,1)$
	such that $\lambda_0 > \lambda_1 > \lambda_2$ and due to Thomas \cite{thomas} we know that
	$$ \lambda_0 \eqqcolon \lambda, \quad \lambda_1 = -\frac{1}{\lambda+1}, 
	\quad \lambda_2 = - \frac{\lambda + 1}{\lambda}. $$
	The roots of $f_n$ do of course depend on $n$ but we omit that notation for better readability.
	
	Levesque and Waldschmidt \cite{lev_wald} ``twisted'' the polynomials by
	an exponent $t$, i.e. they looked at 
	\begin{equation}
		\label{eq-lev-wald}
		F_{n,t}(X,Y) = (X-\lambda^tY) (X-\lambda_1^tY) (X-\lambda_2^tY) 
			= \pm1.
	\end{equation}
	which is irreducible in $\ZZ[X,Y]$.
	The minimal polynomial of $\lambda^t$ is $F_{n,t}(X,1)$.

	Levesque and Waldschmidt proved that for $n$ and $t$ exceeding
	some effectively computable constant, there are no non-trivial 
	solutions. Further, they provided a list of small solutions and raised the
	question whether it contains all solutions.
	
	In this paper, we affirm their question, i.e. we solve the problem for 
	all possible pairs $(n,t)$ and prove that their list is complete.

    It is straightforward to check that
	\begin{align*}
		F_{-n-1,t}(X,Y) &= F_{n,t}(-Y,-X),\\
		F_{n,-t}(X,Y) &= F_{n,t}(-Y,-X),
	\end{align*}
	therefore, by solving the family of equations for $n\geq0$ and $t>0$,
    we actually solve it for all $n\in\ZZ$ and $t\in\ZZ\backslash\{0\}.$
    If $t=1,$ we get \eqref{eq-shanks}, therefore throughout this paper
    we may even assume that $t \geq 2$.

	\begin{theorem}
		\label{theo-results}
		Let $n,t\in\ZZ$ and $t\neq 0$. The only possible solutions $(x,y)\in \ZZ^2$ to \eqref{eq-lev-wald} 
        are the trivial solutions 
        \begin{itemize}
            \item $(x,y)=(0,\pm1)$ and $(x,y)=(\pm1,0)$ for all $(n,t)$,
            \item $(x,y) =(\pm1,\mp1)$ for $t=1$,
        \end{itemize} 
        and the ones given in the following table:
		
		\begin{tabular}{c|ccccccc}
			$(n,t)$ & & $(\pm x,\pm y)$ & & & & & \\
			\hline
            $(0,1)$&$(-9,5)$&$(-4,9)$&$(-2,1)$&$(-1,2)$&$(1,1)$&$(5,4)$&\\
            $(1,1)$&$(-3,2)$&$(1,3)$&$(2,1)$&&&&\\
            $(3,1)$&$(-9,7)$&$(-2,9)$&$(7,2)$&&&&\\
			$(0,2)$&$(1,1)$&$(1,5)$&$(2,1)$&$(3,1)$&$(3,2)$&$(13,4)$&$(14,9)$\\
			$(1,2)$&$(2,1)$&$(3,1)$&$(7,2)$&$(7,3)$&&&\\
			$(2,2)$&$(2,1)$&$(7,1)$&&&&&\\
			$(4,2)$&$(3,2)$&&&&&&\\
            $(0,3)$&$(2,1)$&&&&&&\\
			$(0,5)$&$(3,1)$&$(19,-1).$&&&&&
		\end{tabular}
	\end{theorem}

	
    We prove this Theorem in two main steps:
    \begin{itemize}
        \item In Section~\ref{sec-tw}, we find an upper bound for the parameter $n$.
        To minimize this bound, we use Laurent's work from 2008 \cite{laurent}
        instead of the ``classical'' results of Baker-Wüstholz and Matveev.
        The crucial step is to apply it twice, both on the linear forms in logarithms
        $\Lambda$ and $\Lambda'$, as seen in Lemmas~\ref{lem-bound-L'} and 
        \ref{lem-bound-Lambda}.
        The other important relations are established in 
        Lemma~\ref{lem-B-logy-log-lam-t} and Lemma~\ref{lem-AB-lambda}, namely
        $$|A| n \log n \sim |B| \sim 
            3 \left( \frac{\log|y|}{\log\lambda} + t \right).$$
        \item The bounds for $n$ and $t$ are too large to allow a direct 
        resolution of all associated Thue equations, so in Section~\ref{sec-small}, 
        we introduce additional reduction arguments to 
        further restrict the set of admissible pairs $(n,t)$ for which 
        \eqref{eq-lev-wald} may have non-trivial solutions.
        The remaining Thue equations can then be solved by the standard
        methods.
    \end{itemize}

    The second step requires extensive computer-assisted calculations.
    The code used for these computations is publicly available at
    \url{https://git.sbg.ac.at/b1065846/complete-resolution-of-a-family-of-twisted-thue-equations}.

\section{Auxiliary results}
	
	In this section we provide some useful tools that we will need later on.

	\begin{lemma}
		\label{lem-log-bound}
		Let $z \in \CC$, $z\neq0$, and $a \in (0,1]$ with $|z-1| \leq a$.
		Then
		$$ |\log z| \leq |z-1| \frac{- \log(1-a)}{a}.$$
	\end{lemma}

	\begin{proof}
        We compute
		\begin{align*}
			|\log z| &= 
			\left| \sum_{k=1}^{\infty} 
			\frac{(-1)^{k-1} (z-1)^k}{k} \right| \\
			&\leq |z-1| \sum_{k=1}^{\infty} \frac {a^{k-1}}{k} \\
			&= |z-1| \sum_{k=1}^{\infty} \frac1a \frac{(-1)^{k}(-a)^k}{k} \\
			&= |z-1| \frac{- \log(1-a)}{a} .
		\end{align*}
	\end{proof}

	Choosing $a = \frac12$ we see that $|z-1| \leq \frac12$ implies 
	$$ |\log z| \leq |z-1| \frac{- \log(\frac12)}{\frac12} =
	2 |z-1| \cdot \log 2 \leq 2 |z-1|, $$
	thus we have

	\begin{cor}
		\label{cor-log-12}
		Let $z \in \CC$, $z\neq0$, with $|z-1| \leq \frac12$.		
		Then
		$$|\log z| \leq 2 |z-1|.$$
	\end{cor}

	The following useful lemma is due to Peth\H{o} and de Weger \cite{pethoe-deweger}:
	
	\begin{lemma}
		\label{lem-petho-dew}
		Let $a,b \geq 0$, $k \geq 1$ and $x \in \RR$ be the largest solution
		of \\
        $x = a + b (\log x)^k$. If $b > \left( \frac{e^2}k \right)^k$, then
		$$ x < 2^k \left( a^{\frac 1k} + b^{\frac 1k} 
		\log \left( k^k b \right) 
		\right) ^k $$
		and if $b \leq \left( \frac{e^2}k \right)^k$ then
		$$ x \leq 2^k \left( a^{\frac 1k} + 2 e^2 \right)^k. $$
	\end{lemma}
	\begin{proof}
		See \cite{smart}, Appendix B, Lemma B.1.
	\end{proof}

	Let $\alpha$ be an algebraic number with minimal polynomial
	$$a_d X^d + \dots + a_1 X + a_0$$ and conjugates 
	${\alpha = \alpha^{(1)}, \dots, \alpha^{(d)}}$.
	Then the absolute logarithmic Weil height of $\alpha$ is defined by
	$$ h(\alpha) = \frac 1d \left( \log a_d + 
	\sum_{i = 1}^{d} \log \max \left\{1, \abs{\alpha^{(i)}} \right\}
	\right).$$
	For simplicity, we refer to the absolute logarithmic Weil height simply as the height.
	
	From \cite[Chapter 3]{waldschmidt-dioph-approx} we get some properties 
	for the height.
	\begin{lemma} [Properties of the height]
		\label{lem-height}
		Let $r,s \in \ZZ$ and let $\alpha$ be a non-zero algebraic number.
		Then
		\begin{enumerate} [label = (\alph*)]
			\item $ h(\alpha \pm \beta) \leq \log 2 + h(\alpha) + h(\beta) ,$
			\item $ h(\alpha \beta) \leq h(\alpha) + h(\beta) ,$
			\item $ h(\alpha^{\frac rs}) = \abs{\frac rs} h(\alpha),$
			\item $h(\frac rs) = \max \left\{ \log |r|, \log |s| \right\}, $
			\item $ h(\alpha) = 0$ if and only if $\alpha$ is a root of unity.
		\end{enumerate}
	\end{lemma}
	
	We will need continued fractions, in particular the following lemma
	(see \cite[Theorem 184]{hardy_wright}):
	\begin{lemma}[Legendre]
		\label{theo-convergent}
		Let $x \in \RR$ and $p, q$ positive integers such that
		$$ \abs{x - \frac{p}{q}} < \frac{1}{2q^2}. $$
		Then $\frac pq$ is a convergent of the continued fraction of $x$.
	\end{lemma}

\subsection{Linear forms in logarithms}
\label{sec-lin-log}
	Let $\alpha_1, \dots, \alpha_n$ be algebraic numbers not equal to 0 or 1,
	$K = \QQ(\alpha_1, \dots, \alpha_n)$ and $d = [K:\QQ]$.
	Further let
	$$ \Lambda = b_1 \log \alpha_1 + \dots b_n \log \alpha_n 
	\neq 0 $$
	with $b_i \in \ZZ$.
	We define the modified height
	$$h_m(\alpha) = \max\left\{h(\alpha), \frac{|\log \alpha|}{d}, 
		\frac 1d \right\}.$$
	
	\begin{theorem}[Baker--Wüstholz]
		\label{theo-baker-wust}
		Let ${b_1, \dots, b_n \in \ZZ,}$ and\\
		$B = \max\{|b_1|, \dots, |b_n|\} \geq 3$.		
		Then we have
		$$ \log |\Lambda| > - c \prod_{i=1}^{n} h_m(\alpha_i) \log B $$
		with 
		$$ c = c(n,d) = 18 (n+1)! n^{n+1} (32d)^{n+2} \log (2nd). $$
	\end{theorem}
	

	If we have a linear form that consists of only two logarithms, there is
	a result by Laurent from 2008 \cite{laurent} that gives us a much 
	smaller constant.
	
	Let $\alpha_1, \alpha_2$ be two non-zero algebraic numbers and 
	$b_1, b_2$ positive integers. We look at the linear form in two logarithms
	$$ \Lambda = b_2 \log \alpha_2 - b_1 \log \alpha_1.	 $$
	
	\begin{theorem}
		\label{theo-laurent-2}
		Let $a_1, a_2, h, \rho$ and $\mu$ be real numbers with $\rho>1$ and
		$\frac13 \leq \mu \leq 1$. Set
		\begin{align*}
			\sigma = \frac{1 + 2 \mu - \mu^2}{2}, \quad
			\lambda = \sigma \log \rho, \quad
			H = \frac{h}{\lambda} + \frac{1}{\sigma},\\
			\omega = 2 \left( 1 + \sqrt{1 + \frac{1}{4 H^2}} \right), \quad
			\theta = \sqrt{1 + \frac{1}{4 H^2}} + \frac{1}{2H}.
		\end{align*}
		Consider the linear form $\Lambda = b_2 \log \alpha_2 - b_1 \log \alpha_1$, 
		where $b_1$ and $b_2$ are positive integers.
		Suppose that $\alpha_1$ and $\alpha_2$ are multiplicatively independent.
+		Set\\ $D~=~[\QQ(\alpha_1, \alpha_2) \colon \QQ] /
		[\RR(\alpha_1,\alpha_2)\colon\RR]$ and assume that
		\begin{enumerate}[label=(\roman*)]
			\item $h \geq \max \left\{ 
				D \left(\log \left(\frac{b_1}{a_2} + \frac{b_2}{a_1}\right) 
					+ \log \lambda + 1.75 \right) + 0.06,
				\lambda, 
				\frac{D \log 2}{2} \right\}$
			\item $a_i \geq \max \{ 1, \rho |\log \alpha_i| - \log |\alpha_i|
					+ 2 D h(\alpha_i) \} \quad (i=1,2)$
			\item $a_1a_2 \geq \lambda^2.$
		\end{enumerate}
		Then
		$$\log|\Lambda| \geq - C \left( h + \frac{\lambda}{\sigma}\right)^2 a_1 a_2
			- \sqrt{\omega\theta} \left(h + \frac{\lambda}{\sigma}\right)
			- \log \left(C' \left(h + \frac{\lambda}{\sigma}\right)^2 a_1 a_2\right) $$
		with
		\begin{align*}
			C &= \frac{\mu}{\lambda^3 \sigma} \left(\frac{\omega}{6} + \frac 12 
				\sqrt{\frac{\omega^2}{9} + \frac{8 \lambda\omega^{\frac54} \theta^{\frac14}} 
					{3 \sqrt{a_1a_2}H^{\frac12}}
					+ \frac43 \left(\frac{1}{a_1} + \frac{1}{a_2}\right)
					\frac{\lambda\omega}{H}
				}\right)^2,\\
			C' &= \sqrt{\frac{C\sigma\omega\theta}{\lambda^3\mu}}.
		\end{align*}
	\end{theorem}
	
	Let us set
	$$ b' = \frac{b_1}{d \log A_2} + \frac{b_2}{d \log A_1} $$
	with real numbers $A_1, A_2 > 0$ that fulfill $\log A_i \geq h_m(\alpha_i)$.
	
	We define the constants $C_1(m)$ and $C_2(m)$ for $m = 10, 12, \dots, 30$ as suggested in Table~\ref{tabl-const-laurent}.
	
	\begin{table}[!h]
		\begin{center}
			\caption{Laurent's constants}
			\label{tabl-const-laurent}
			\begin{tabular}{ c| ccccccccccc} 
				$m$ & 10 & 12 & 14 & 16 & 18 & 20 & 22 & 24 & 26 & 28 & 30 \\
				\hline
				$C_1$ & 32.3 & 29.9 & 28.2 & 26.9 & 26.0 & 25.2 & 24.5 & 24.0 &
				23.5 & 23.1 & 22.8 \\
				$C_2$ & 25.2 & 23.4 & 22.1 & 21.1 & 20.3 & 19.7 & 19.2 & 18.8 &
				18.4 & 18.1 & 17.9 \\
			\end{tabular}
		\end{center}
	\end{table}
	
	Then we get from \cite[Corollary 2]{laurent}
	
	\begin{cor}
		\label{cor-laurent-2}
		Suppose $\alpha_1$ and $\alpha_2$ are multiplicatively independent.
		If $\alpha_1,$ $\alpha_2,$ $\log \alpha_1,$ $\log \alpha_2$ are real and
		positive, then
		$$ \log |\Lambda| \geq - C_2(m) d^4 \left( 
		\max \left\{\log b' + 0.38, \frac md, 1 \right\} \right)^2
		\log A_1 \log A_2 $$
		where $(m, C_2(m))$ is a pair from Table~\ref{tabl-const-laurent}.
	\end{cor}

To improve readability in the subsequent sections, we introduce the $L$-notation.
Assume that $f(x),g(x)$ and $h(x)$ are real functions and $h(x)>0$. We will write
\[f(x)=g(x)+L_c (h(x))\]
for
\[g(x)-h(x)\leq f(x) \leq g(x)+h(x).\]

The $L$-notation is similar to the $O$-notation, with the advantage that the error term is given explicitly. For more details on the $L$-notation we refer to \cite[Section~3.1]{heuberger-togbe-ziegler}.

\section{Twisted Thue Equation}
\label{sec-tw}

	We denote the number field generated by the roots of $f_n$ by
	$K_n = \QQ(\lambda)$. 
	It is well-known (see for example \cite{thomas1979}) that $K_n/\QQ$ is 
    Galois and that either
    pair of the three roots $\lambda, \lambda_1, \lambda_2$
	generates the group of units in $\ZZ[\lambda]$.
	We choose $\{\lambda, \lambda_2\}$.
	
	Further, we write  $\alpha_i = \lambda_i^t$ and
	$\beta_i = x - \lambda_i^t y$.
	Then $\alpha_i = \sigma_i(\alpha)$ and $\beta_i = \sigma_i(\beta_0)$,
    where $\sigma_i$ indicates the $\QQ$-automorphisms in $\Gal(K_n/\QQ)$

	Let $(x,y)$ be a solution of \eqref{eq-lev-wald}.
    We see that independently of $n$ and $t$, $y = 0$ implies $x= \pm 1$ and 
    $x=0$ implies $y=\pm1.$
	Thus, we call $(x,y)= (\pm 1,0), (0,\pm1)$ the 
    trivial solutions of \eqref{eq-lev-wald}.
    In the special case $t=1$, the solutions $(\pm1,\mp1)$ are also trivial 
    as they solve \eqref{eq-lev-wald} for every $n$.
    Throughout this paper, we restrict our attention to non-trivial solutions
    and therefore assume that $|y|\geq 1$.
	
	Let us choose $j\in \{0,1,2\}$ such that
	$$|\beta_j| = \min_{i = 0,1,2} \left\{ |\beta_i| \right\}.$$
	Then we say that the solution $(x,y)$ is of type $j$.
	We set $\beta_k = \sigma_1(\beta_j),\ \beta_l = \sigma_2(\beta_j)$.
	In the following, when computing bounds, we must consider all 
    three cases $j=0,1,2$, resulting in three different bounds.
	For the sake of brevity, we always replace these constants by 
	their maximum.
	Accordingly, when we refer to the ``worst case'', we mean this 
	largest possible constant, which is then used in the subsequent arguments.
	
	For technical reasons, we assume that $n\geq 2$ in the following considerations. 
    For the cases $n=0$ and $n=1$ we refer to Section~\ref{sec-n01}

	\begin{lemma}
		\label{lem-bound-bj}
        We have
		\begin{align*}
			\abs{\beta_j} 
			&\leq \frac{4}{|y|^2 \abs{\alpha_j - \alpha_k} \abs{\alpha_j - \alpha_l}}\\
			&\leq \frac{c_1}{|y|^2}
			 \cdot 
			\begin{cases}
				\abs{\lambda}^{-2t} &\text{if }j=0,\\
				\abs{\lambda \lambda_2}^{-t} &\text{if }j=1,2
			\end{cases}		
		\end{align*}
		with
		$$ c_1 = 4 \cdot 
		\left( 1 - \abs{\frac{\lambda_2}{\lambda}}^t \right)^{-1}
		\left( 1 - \abs{\frac{\lambda_1}{\lambda_2}}^t \right)^{-1}.$$
		
	\end{lemma}
	\begin{proof}
		As $j$ was chosen such that $\abs{\beta_j}$ is minimal, we get
		\begin{align}\label{eq-beta-y2}
			\begin{split}
			\abs{\beta_i} &\geq \frac 12 \left(|\beta_j| + |\beta_i|\right)\\
			&\geq \frac 12 \abs{\beta_j - \beta_i}\\
			&= \frac 12 \abs{x - \alpha_i y - x + \alpha_j y}\\
			&= \frac 12 |y| \abs{\alpha_i - \alpha_j}.
			\end{split}
		\end{align}
		
		Further we know that 
		$$\beta_j \beta_k \beta_l = F_{n,t}(x,y) = \pm1,$$
		and together this gives us the stated inequality with
		\begin{align*}
			c_1 
			&= 4 \cdot \begin{cases}
				\left( 1 - \abs{\frac{\lambda_1}{\lambda}}^t \right)^{-1}
				\left( 1 - \abs{\frac{\lambda_2}{\lambda}}^t \right)^{-1}
				& \text{if } j=0,\\
				\left( 1 - \abs{\frac{\lambda_1}{\lambda}}^t \right)^{-1}
				\left( 1 - \abs{\frac{\lambda_1}{\lambda_2}}^t \right)^{-1}
				& \text{if } j=1,\\
				\left( 1 - \abs{\frac{\lambda_2}{\lambda}}^t \right)^{-1}
				\left( 1 - \abs{\frac{\lambda_1}{\lambda_2}}^t \right)^{-1}
				& \text{if } j=2.
			\end{cases}
		\end{align*}
		We observe that $c_1$ is maximized when $j=2$, which constitutes 
		the worst-case scenario.
		Therefore we will use that value in our computations, i.e.
		$$ c_1 = 4 \cdot 
		\left( 1 - \abs{\frac{\lambda_2}{\lambda}}^t \right)^{-1}
		\left( 1 - \abs{\frac{\lambda_1}{\lambda_2}}^t \right)^{-1}.$$
	\end{proof}

	\begin{lemma}
		For $i\neq j$ we have
		$$ \log \abs{\beta_i} = \log |y| + t \begin{cases}
			\log \lambda &\text{if }\{i,j\} \neq \{1,2\}\\
			\log |\lambda_2| &\text{if } \{i,j\}=\{1,2\}
		\end{cases} + L(c_2) $$
		with $c_2 = \log 2 + \log \left(1-\abs{\frac{\lambda_1}{\lambda}}^t + \frac{c_1}{|y|^2 |\lambda\lambda_2|^t} \right)$.

	\end{lemma}
	\begin{proof}

		For $i\neq j$, we know from \eqref{eq-beta-y2}
		\begin{align*}
			\abs{\beta_i} &\geq \frac{|y|}{2} \abs{\alpha_i - \alpha_j}\\
            &= \frac{|y|}{2} \cdot \begin{cases}
				\lambda^t \left(1 - \abs{\frac{\lambda_{i}}{\lambda}}^t\right)
				&\text{if } j=0\\
                \lambda^t \left(1 - \abs{\frac{\lambda_{j}}{\lambda}}^t\right)
				&\text{if } i=0\\
				\lambda_2^t \left(1 - \abs{\frac{\lambda_1}{\lambda_2}}^t\right)
				&\text{if } \{i,j\}=\{1,2\}
			\end{cases}
		\end{align*}
        thus
		\begin{align*}
			\log |\beta_i| \geq \log |y| + t \cdot \begin{cases}
				\log \lambda &\text{if } \{i,j\} \neq \{1,2\} \\
				\log |\lambda_2| &\text{if } \{i,j\} = \{1,2\} 
			\end{cases} + \kappa_0
		\end{align*}
		with 
		\begin{align*}
			\kappa_0 = - \log 2 + \begin{cases}
				\log \left(1 - \abs{\frac{\lambda_{i}}{\lambda}}^t\right) &\text{if } j=0 \\
                \log \left(1 - \abs{\frac{\lambda_{j}}{\lambda}}^t\right) &\text{if } i=0 \\
				\log \left(1 - \abs{\frac{\lambda_{1}}{\lambda_2}}^t\right) &\text{if } \{i,j\} = \{1,2\} 
			\end{cases}
		\end{align*}
		and 
		\begin{align*}
			\abs{\beta_i} &\leq |y| \abs{\alpha_i - \alpha_j} + \abs{\beta_j}\\
			&< |y| \cdot \begin{cases}
				\lambda^t \cdot \left(1-\abs{\frac{\lambda_i}{\lambda_0}}^t + \frac{c_1}{|y|^2 \lambda^{2t}} \right)
				&\text{if } j=0\\
				\lambda^t \cdot \left(1-\abs{\frac{\lambda_j}{\lambda_0}}^t + \frac{c_1}{|y|^2 |\lambda\lambda_2|^t} \right)
				&\text{if } i=0\\
				|\lambda_2|^t \cdot \left(1-\abs{\frac{\lambda_1}{\lambda_2}}^t + \frac{c_1}{|y|^2 |\lambda\lambda_2|^t} \right)
				&\text{if } \{i,j\}=\{1,2\}.
			\end{cases}
		\end{align*}
		This leads us to
		\begin{align*}
			\log \abs{\beta_i} \leq \log |y| + t \begin{cases}
				\log \lambda &\text{if } \{i,j\} \neq \{1,2\} \\
				\log |\lambda_2| &\text{if } \{i,j\} = \{1,2\} 
			\end{cases}
			+ \kappa_1.
		\end{align*}

        In total we get
        \begin{align*}
            \log \abs{\beta_i} = \log |y| + t \begin{cases}
    			\log \lambda &\text{if }\{i,j\} \neq \{1,2\}\\
    			\log |\lambda_2| &\text{if } \{i,j\}=\{1,2\}
		      \end{cases} + L(c_2)
        \end{align*}
        where $c_2 \geq \max\{\kappa_0, \kappa_1\},$ i.e. we look at the worst case ($j=1,\ i=0$) and set
        $$ c_2 = \log 2 + \log \left(1-\abs{\frac{\lambda_1}{\lambda}}^t 
            + \frac{c_1}{|y|^2 |\lambda\lambda_2|^t} \right).$$
	\end{proof}

	Now we look at Siegel's identity
	$$ \beta_j (\alpha_k - \alpha_l) + \beta_k (\alpha_l - \alpha_j) 
	+ \beta_l (\alpha_j - \alpha_k) = 0 $$
	and rewrite it as
	$$ \frac{\beta_j}{\beta_l}  \frac{\alpha_k - \alpha_l}{\alpha_k - \alpha_j} +
	\frac{\beta_k}{\beta_l}  \frac{\alpha_l - \alpha_j}{\alpha_k - \alpha_j}
	= 1. $$
	
	For the first summand we obtain
	\begin{align}
		\label{eq-bound-siegel}
		\begin{split}
		\abs{\frac{\beta_j}{\beta_l}  \frac{\alpha_k - \alpha_l}{\alpha_k - \alpha_j}} 
		&\leq \abs{\frac{\alpha_k - \alpha_l}{\alpha_k - \alpha_j}} 
			\cdot \frac{2}{\abs{\alpha_l - \alpha_j} |y|} 
			\cdot \abs{\beta_j} \\
		&= \frac{{c_3}}{|y|^3} 	\begin{cases}
			\abs{\lambda}^{-3t} &\text{if }j=0,\\
			\abs{\lambda \lambda_2^2}^{-t} &\text{if }j=1,2
		\end{cases}		
		\end{split}
	\end{align}
	with
	\begin{align*}
		c_3 
		&= \frac{c_1^2}{2} \cdot \begin{cases}
			\abs{\lambda_2/\lambda}^t \cdot\left(1 - \abs{\frac{\lambda_1}{\lambda_2}}^t \right) &\text{if } j=0,\\
			\left(1 - \abs{\frac{\lambda_2}{\lambda}}^t\right) &\text{if } j=1,\\
			\left(1 - \abs{\frac{\lambda_1}{\lambda}}^t\right) &\text{if } j=2
		\end{cases}\\
		&< \frac{c_1^2}{2}.
	\end{align*}

	We define $e^\Lambda = \frac{\beta_k}{\beta_l}  \frac{\alpha_l - \alpha_j}{\alpha_k - \alpha_j}$, thus
	$$\Lambda = \log \beta_k - \log \beta_l 
		+ \log \abs{ \frac{\alpha_l - \alpha_j}{\alpha_k - \alpha_j}}. $$
	Since $|y|\geq 1$ and $t\geq2$, Inequality~\eqref{eq-bound-siegel} yields 
        the absolute constant 
		$$\abs{1-e^\Lambda} < c_3 \abs{\lambda\lambda_2^2}^{-2}= c_4.$$
	
	\begin{lemma}\label{lem-bound-Lambda}
	We have
		$$ \abs{\Lambda} < \frac{c_5}{|y|^3} \abs{\lambda\lambda_2^2}^{-t},$$
    where $c_5 = \frac{-\log(1-c_4)}{c_4} c_3$.
	\end{lemma}
	
	\begin{proof}
		Since $\abs{1-e^{\Lambda}} < c_4 < 1$, we can use Lemma~\ref{lem-log-bound} to 
		get an upper bound for $\Lambda$,
		\begin{align*}
			|\Lambda| &=\abs{\log \left(1 + 1-e^{\Lambda}\right)} \\
			&\leq \frac{-\log(1-c_4)}{c_4} \abs{1-e^{\Lambda}}\\
			&\leq \frac{-\log(1-c_4)}{c_4} c_3 \abs{y}^{-3}
			\cdot \begin{cases}
				\lambda^{-3t} &\text{if } j=0,\\
				\abs{\lambda\lambda_2^2}^{-t} &\text{if } j=1,2,
			\end{cases}\\
			&\leq c_5 \abs{y}^{-3}
			\cdot \begin{cases}
				\lambda^{-3t} &\text{if } j=0,\\
				\abs{\lambda\lambda_2^2}^{-t} &\text{if } j=1,2.
			\end{cases}
		\end{align*}
	\end{proof}

	Since $N_{K/\QQ}(\beta_i) = \pm 1$, we know that $\beta_1$, $\beta_2$, $\beta_3$
    are units in the corresponding ring of integers. 
	Thus, according to Dirichlet's unit theorem, we can write
	$$\beta_j = \lambda^{a}\lambda_2^{b}.$$
	With algebraic conjugation, we get
	\begin{align*}
		\beta_k &= \lambda_1^a \lambda^b = \lambda^{b-a} \lambda_2^{-a}, \\ 
		\beta_l &= \lambda_2^a \lambda_1^b = \lambda^{-b} \lambda_2^{a-b},
	\end{align*}
	which leads us to
	\begin{align*}
		\Lambda &= \log |\beta_k| - \log |\beta_l|
			+ \log \abs{\frac{\alpha_l - \alpha_j}{\alpha_k - \alpha_j}}\\
		&= (2b-a) \log |\lambda| + (b-2a) \log |\lambda_2|
		 + \log \abs{\frac{\alpha_l - \alpha_j}{\alpha_k - \alpha_j}}.
	\end{align*} 
	
	We set
	$$\Lambda' = A \log |\lambda| + B \log |\lambda_2|$$
	with 
	\begin{align*}
		A = \begin{cases}
			2b-a, \\
			2b-a+t, \\
			2b-a-t
		\end{cases}
		\ \text{and}\quad 
		B = \begin{cases}
			b-2a &\quad\text{if } j=0,\\
			b-2a-t &\quad\text{if } j=1,\\
			b-2a+t &\quad\text{if } j=2.
		\end{cases}
	\end{align*}
	
	\begin{lemma} \label{lem-B-logy-log-lam-t}
	The following inequalities hold:
		\begin{align}
			\label{eq-bound_B}
			c_8' \left(\frac{\log|y|}{\log \lambda} + t\right)
			< |B| < c_8 \left(\frac{\log|y|}{\log \lambda} + t\right).
		\end{align}
	\end{lemma}
	
	\begin{proof}
		We have the linear system of equations
		\begin{align*}
			\begin{pmatrix}
				\log |\beta_k| \\
				\log |\beta_l|
			\end{pmatrix}
			=
			\begin{pmatrix}
				\log |\lambda_1| & \log |\lambda| \\
				\log |\lambda_2| & \log |\lambda_1|
			\end{pmatrix}
			\begin{pmatrix}
				a \\ b
			\end{pmatrix}
			\eqqcolon R
			\begin{pmatrix}
				a \\ b
			\end{pmatrix}.
		\end{align*}
		Solving this linear system of equations,
		we can compute bounds for $a$ and $b$.
		These bounds depend on $j$. We give the details for the case $j=0$. The other cases can be dealt with similarly. We note that in the other cases we get tighter bounds for $|B|$.
		
		We have
		$$\det R = \log (\lambda+1) ^2 - \log (\lambda+1) \log (\lambda) + (\log \lambda)^2, $$
		ergo
		$$\left(\log (\lambda+1)\right)^2 > \det R > (\log \lambda)^2.$$
		
		By inverting the matrix $R$, we get
		\begin{align*}
			a &= \frac{1}{\det R} \left( \log |\lambda_1| \log |\beta_k| - \log |\lambda| \log |\beta_l|\right),\\
			b &= \frac{1}{\det R} \left(- \log |\lambda_2| \log |\beta_k| + \log |\lambda_1| \log |\beta_l|\right).
		\end{align*}
		
		For $j=0$ we get 
		\begin{align}
			\label{eq-comp-a-b}
			\begin{split}
			|a| &= \frac{1}{\det R} \left(\log(\lambda+1) + \log\lambda\right)
			\cdot \left(\log |y| + t \log \lambda + L(c_2) \right) \\
			&= \frac{\log\lambda \left( \log(\lambda+1) + \log \lambda\right) }{\det R}
			\left(1 + \frac{L(c_2)}{2 \log \lambda}\right)
			\left(\frac{\log |y|}{\log \lambda} + t\right) \\
			&=c_6 \left(1 + \frac{L(c_2)}{2 \log \lambda}\right) \left(\frac{\log |y|}{\log \lambda} + t\right) ,\\
			|b| &= \frac{1}{\det R} \left( 2 \log(\lambda+1) - \log\lambda\right) 
			\cdot \left( \log |y| + t \log \lambda + L(c_2) \right) \\
			&=c_7 \left(1 + L\left(\frac{c_2}{2 \log \lambda}\right)\right) \left( \frac{\log|y|}{\log \lambda} +t \right) ,
			\end{split}
		\end{align}
		and, since $B=b-2a$, we get 
		\begin{align*}
			c_8' \left(\frac{\log|y|}{\log \lambda} + t\right)
			< |B| < c_8 \left(\frac{\log|y|}{\log \lambda} + t\right),
		\end{align*}
        where 
        $$c_8 = c_7\cdot \left( 1+ \frac{c_2}{2 \log \lambda}\right) - 2c_6 \cdot \left( 1- \frac{c_2}{2 \log \lambda}\right) $$ 
        and
        $$c_8' = c_7\cdot \left( 1- \frac{c_2}{2 \log \lambda}\right) - 2c_6 \cdot \left( 1+ \frac{c_2}{2 \log \lambda}\right) .$$
                

        The cases $j=1$ and $j=2$ yield the same result by analogous computations.
		\end{proof}
		
	\begin{lemma}
    \label{lem-bound-L'}
		We have
		$$\abs{\Lambda'} < c_9 \abs{\lambda\lambda_2^2}^{-t},$$
        where $ c_9 = c_5 +1 + \abs{\log \left(1 - \left(\frac{\lambda_2}{\lambda}\right)^t\right)} 
            \abs{\lambda\lambda_2^2}^t.$
	\end{lemma}
	\begin{proof}
		For $j=0$ 
		\begin{align*}
			\log \abs{\frac{\alpha_2 - \alpha_0}{\alpha_1 - \alpha_0} }
			&\leq \log \left(1 - \left(\frac{\lambda_2}{\lambda}\right)^t\right)
			- \log \left(1 - \left(\frac{\lambda_1}{\lambda}\right)^t\right)\\
			&\leq \abs{\frac{\lambda_2}{\lambda}}^t + \abs{\frac{\lambda_1}{\lambda}}^t,
		\end{align*}
		thus
		\begin{align*}
			|\Lambda'| &\leq c_5 |y|^{-3} \frac{1}{\lambda^{3t}} + 
			\abs{\frac{\lambda_2}{\lambda}}^t + \abs{\frac{\lambda_1}{\lambda}}^t\\
			&\leq c_9 \abs{\lambda \lambda_2^2}^{-t},
		\end{align*}
        where
        \begin{align*}
            c_9 = c_5 \abs{\frac{\lambda_2}{\lambda}}^{2t} 
                + |\lambda_2|^{3t} 
                + \abs{\frac{\lambda_2}{\lambda}}^t.
        \end{align*}
		
		For $j=1$, we get 
		\begin{align*}
			\log \abs{\frac{\alpha_0 - \alpha_1}{\alpha_2 - \alpha_1} } 
			&= t \log |\lambda| - t \log |\lambda_2| 
			+ \log \left(1 - \left(\frac{\lambda_1}{\lambda}\right)^t\right)
			- \log \left(1 - \left(\frac{\lambda_1}{\lambda_2}\right)^t\right)\\
			&\leq t \log |\lambda| - t \log |\lambda_2| 
			+ \abs{\frac{\lambda_1}{\lambda}}^t 
			+ \kappa_1 \abs{ \lambda \lambda_2^2}^{-t},
		\end{align*}
        with 
        $\kappa_1 = \abs{\log \left(1 - \abs{\frac{\lambda_1}{\lambda_2}}^t\right)} 
            \abs{\lambda\lambda_2}^t.$
        
		Noting that $\lambda_1 =\frac{1}{\lambda\lambda_2}$, we consequently get
		\begin{align*}
			\abs{\Lambda'} &\leq c_5 |y|^{-3} \abs{\lambda\lambda_2^2}^{-t}
			+ \abs{\frac{\lambda_2}{\lambda}\cdot \frac{1}{\lambda \lambda_2^2}}^t
			+  \kappa_1 \abs{\lambda \lambda_2^2}^{-t} \\
			&\leq c_9 \abs{\lambda\lambda_2^2}^{-t},
		\end{align*}
        with $c_9 = c_5 + \abs{\frac{\lambda_2}{\lambda}}^t + \kappa_1.$
		
		For $j=2$
		\begin{align*}
			\log \abs{\frac{\alpha_1 - \alpha_2}{\alpha_0 - \alpha_2} } 
			&= t \log |\lambda_2| - t \log |\lambda| 
			+ \log \left(1 - \left(\frac{\lambda_1}{\lambda_2}\right)^t\right)
			- \log \left(1 - \left(\frac{\lambda_2}{\lambda}\right)^t\right)\\
			&\leq t \log |\lambda_2| - t \log |\lambda| 
			+ \abs{\frac{\lambda_1}{\lambda_2}}^t 
            + \kappa_2 \abs{\lambda \lambda_2^2}^{-t} ,
		\end{align*}
        where 
        $\kappa_2 = \abs{\log \left(1 - \left(\frac{\lambda_2}{\lambda}\right)^t\right)} 
            \abs{\lambda\lambda_2^2}^t.$
		Consequently, we get
		\begin{align*}
			\abs{\Lambda'} &\leq c_5 |y|^{-3} \abs{\lambda\lambda_2^2}^{-t}
			+ \abs{\lambda\lambda_2^2}^{-t}
			+ \kappa_2 \abs{\lambda \lambda_2^2}^{-t} \\
			&\leq c_9 \abs{\lambda\lambda_2^2}^{-t},
		\end{align*}
        where
        $c_9 = c_5 + 1 + \kappa_2.$
        
        It is easy to see that $c_9$ is maximal in the case $j=2$, thus we will use this bound 
        in the subsequent steps.
	\end{proof}
	
	\begin{lemma}
		\label{lem-AB-lambda}
		For $n\geq 1$, we have
		$$\lambda \log \lambda |A| - \frac{c_9}{\lambda}
			< |B| < 
			(\lambda+1) \log \lambda |A| + \frac{c_9}{\lambda} $$
        and consequently 
        $$|A| < |B|.$$
	\end{lemma}
	\begin{proof}
		Since
		\begin{align*}
			\abs{A \log \lambda + B \log |\lambda_2|} &= |\Lambda'|
			< \frac{c_9}{|\lambda\lambda_2|^t}
		\end{align*}
		and
		$$ \frac{1}{\lambda+1} < \log |\lambda_2| < \frac{1}{\lambda}, $$
		we get
		\begin{align*}
			|A| \lambda \log \lambda - |B| &< \frac{c_9}{\lambda},\\
			|B| \frac{\lambda}{\lambda+1} - \lambda \log \lambda |A| 
			&< \frac{c_9}{\lambda}.
		\end{align*}
		We note that since $\lambda>n\geq1$, we have
		$\frac{\lambda+1}{\lambda^2} < \frac{1}{\lambda-1}.$
		Thus, we get the stated inequality.
	\end{proof}

\subsection{Finding an upper bound}
	In this section, we derive upper bounds for $n$ and $t$.
	In the next section, we shall address the remaining cases.
		
	We will need some estimates for $\lambda_i$ and $\log|\lambda_i|$.
	\begin{lemma}
		\label{lem-est-lambda}
		We have
		\begin{align*}
			n + \frac2{n-1} - \frac{3}{(n-1)^2} &< \lambda < n + \frac 2{n-1}; \\ 
			- \frac{1}{n+1} &< \lambda_1 = - \frac{1}{\lambda + 1} 
			< - \frac{1}{n+2}; \\
			- 1 - \frac 1n &< \lambda_2 =  - \frac{\lambda + 1}{\lambda}
			< - 1 - \frac{1}{n+1}.
		\end{align*}
	\end{lemma}
	\begin{proof}
		See \cite{lettl-petho-voutier}.
	\end{proof}

	This yields 

	\begin{cor}
		\label{lem-bounds-loglambda}
		We have
		\begin{align*}
			\log n &< \log \lambda < \log (n+1) + \log n + \frac{1}{n^2};\\
			- \log n - \frac 2n &< \log |\lambda_1| < - \log n - \frac 2n;\\
			\frac 1{n} - \frac{2}{n^2} &< \log |\lambda_2| < \frac 1n + \frac{1}{n^2}.
		\end{align*}
	\end{cor}
	
	We assume that $n \geq 10^5$ and apply the method from Chapter~\ref{sec-tw}
    to compute the constants $c_1,\dots,c_9$. 
    It is sufficient to compute the constants for $n=10^5$ and $t=2$, 
    as they decrease with increasing $n,t$. Therefore, we obtain
    \begin{alignat*}{3}
		c_1 &< 4.0001, & c_6 &< 2.061, \\
		c_2 &< 0.7, & c_7 &< 1.031, \\
		c_3 &< 8.0001, & c_8' &> 2.849, \\
		c_4 &< 8 \cdot10^{-10}, \qquad &c_8 &< 3.151, \\
		c_5 &< 8, & c_9 &< 10.0001.
	\end{alignat*}
    
	Looking at $c_8'$ and $c_8$, we get
	\begin{equation} \label{eq-large-B-bound}
		2.849 \left(\frac{\log|y|}{\log n} + t\right) 
		< B <  3.151 \left(\frac{\log|y|}{\log n} + t\right).
	\end{equation} 
	
	Further, we have
	\begin{align}
		\label{eq-large-el-bounds-bj}
			\abs{\beta_j} &< \frac{c_1}{|y|^2 \abs{\lambda\lambda_2}^{-t}} 
			< \frac{4.0001}{|y|^2\cdot n^t},
	\end{align}
    \begin{align}
		\label{eq-large-el-bounds-L}
		\abs{\Lambda} &\leq \frac{c_5}{|y|^3 \abs{\lambda\lambda_2^2}^t} 
			< \frac{8}{|y|^3 n^t}
	\end{align}
    and
    \begin{align}
		\label{eq-large-el-bounds-L'}
		\abs{\Lambda'} &< \frac{c_9}{\abs{\lambda\lambda_2^2}^t} <\frac{10.0001}{n^t}.
    \end{align}
	
	In addition, it is easy to see that
	\begin{equation} \label{eq-height-large}
		h_m(\lambda_i) = h(\lambda_i) = \frac 13 \abs{\log |\lambda_1|}
			< \frac{1.0001}{3} \log n.
	\end{equation}
	\begin{lemma}
		\label{lem-est-AB-nlogn}
		We have
		$$ 0.9999\cdot n \log n |A| < |B| < 1.01 \cdot n \log n |A| .$$
	\end{lemma}
	\begin{proof}
		Follows from Lemma~\ref{lem-AB-lambda}, since $c_9<10.01$.
	\end{proof}

	Now, we can apply Laurent's result (Corollary~\ref{cor-laurent-2}) to the linear form $\Lambda'$.
	\begin{lemma}
		\label{lem-L'-laurent}
		Assume that $n\geq 10^5$ and $t\geq 2$. Then we have
		$$\log |\Lambda'| > -173 \left( \log \left( \frac{B}{\log n} \right)\right)^2 (\log n)^2 $$
        and
		$$\frac{t}{\log n} < 173 \left( \log \left( \frac{B}{\log n} \right)\right)^2 .$$
	\end{lemma}

	\begin{proof}
		We will apply Laurent's Corollary~\ref{cor-laurent-2} on
		$$ \Lambda' = A \log |\lambda| + B \log |\lambda_2| .$$
		
		First, we need to show that $\Lambda' \neq 0$.
		Since $\lambda$ and $\lambda_2$ are fundamental units, they are
        multiplicatively independent.
        Therefore $\Lambda' = 0$ if and only if 
		$A = B = 0$, which implies $a=b=0$, thus $x= \pm1$ and
		$y = 0$. 
		Since we assumed for $(x,y)$ to be a non-trivial solution,
		we can conclude that $\Lambda'\neq 0$.
		
		We set
		\begin{align*}
			b' &= \frac{A}{1.0001 \log n} + \frac{B}{1.0001 \log n} \\
			&< \frac{\left(1 + \frac{1}{0.99n \log n}\right) B}{1.0001 \log n}\\
			&< \frac{B}{\log n}.
		\end{align*}
		
		Next, we apply Laurent's Corollary~\ref{cor-laurent-2} with $m =30$ and
		$C_2 =17.9$, which yields
		$$\log |\Lambda'| > - 17.9 \cdot3^4 \cdot 
		\left( \max \left\{ \log \left( \frac{B}{\log n} \right) + 0.38,
		10 \right\} \right)^2
		\left( \frac{1.0001}{3}\log n \right)^2.$$
		
		Since
		\begin{align*}
			\log \left(\frac{B}{\log n}\right) > \log(0.99 n) > \log (9.9\cdot10^4) > 11.5 > 10,
		\end{align*}
		we may assume that the maximum in the above inequality is
		$\log \left( \frac{B}{\log n} \right)$. Hence, we get
		\begin{align*}
			\left( \log \left( \frac{B}{\log n} \right) + 0.38 \right)^2 
			&< \left( \frac{11.5 +0.38}{11.5} \right)^2 
			\left(\log \left( \frac{B}{\log n} \right) \right)^2 \\
			&< 1.07 \cdot \left(\log \left( \frac{B}{\log n} \right) \right)^2.
		\end{align*}
		
		This leaves us with
		\begin{align*}
			\log |\Lambda'| &>
			- 17.9\cdot 3^4 \cdot 1.07 \cdot \left(\frac{1.0001}{3}\right)^2
			\left(\log \left( \frac{B}{\log n} \right) \right)^2 (\log n)^2\\
			&> - 172.5 \left(\log \left( \frac{B}{\log n} \right) \right)^2 (\log n)^2.
		\end{align*}
		
		Further, we have the upper bound
		$$\log |\Lambda'| < \log (10.01) - t \log(n).$$
		
		Combining the preceding inequalities, we obtain
		$$ \frac{t}{\log n} 
		< 173 \cdot \left( \log \left( \frac{B}{\log n} \right) \right)^2. $$
	\end{proof}
	
	\begin{lemma}
		\label{lem-t>logylogn}
		If $t> \frac{\log |y|}{\log n}$, then we have
		$$ n \leq 157442. $$
	\end{lemma}
	\begin{proof}
		We note that 
		$B < 3.151 \left(\frac{\log|y|}{\log n} + t\right) < 6.302 t.$
		We use Lemma~\ref{lem-L'-laurent} to get
		\begin{align*}
			\frac{t}{\log n} < 173 \left(\log \frac{6.302 t}{\log n}\right)^2.
		\end{align*}
		We solve the inequality $x<173(\log 6.302x)^2$ for $x=\frac{t}{\log n}$ and get
		$$t < 24733 \log n.$$
		Since we know that
		$$0.99 n\log n < B < 6.302 t,$$
		we get
		$$ n < \frac{6.302}{0.99} \cdot 24733 < 157442.$$
	\end{proof}
	
	As we will see in Section~\ref{sec-small}, the bound on $n$ of Lemma~\ref{lem-t>logylogn} 
    is sufficiently small to solve all potential Thue equations.
    Therefore, we may assume for the remainder of the current section that
	$$t < \frac{\log |y|}{\log n}.$$
	
	We want to apply Laurent's Corollary~\ref{cor-laurent-2} to the linear form in three logarithms
	$$ \Lambda = A \log |\lambda| + B \log |\lambda_2| + \log |\delta|, $$
	with
	\begin{align*}
		\delta =
		\begin{cases}
			\frac{\alpha - \alpha_2}{\alpha - \alpha_1} &\text{if } j = 0,\\
			\frac{1 - \left(\frac{\lambda_j}{\lambda}\right)^t}
			{1 - \left(\frac{\lambda_1}{\lambda_2}\right)^t}
			&\text{if }j = 1,2.
		\end{cases}
	\end{align*}
	To achieve that, we have to transform $\Lambda$ into a linear form in only two logarithms.
    We define
	$$\varphi = \lambda^A \cdot \delta$$
	and we get
	$$\Lambda = \log |\varphi| + B \log |\lambda_2|.$$
	
	\begin{lemma}
	We have
		$$h(\varphi) \leq  1.73 \max\{A, t\} \log n.$$
	\end{lemma}
	\begin{proof}
		First, we compute $h(\delta)$.
		For $j=0$ we obtain
		\begin{align*}
			h(\delta) =
			h \left(\frac{\alpha - \alpha_2}{\alpha - \alpha_1}\right) 
			& \leq 4 t \cdot h(\lambda) + 2 \log 2\\
			&\leq 4.181 \cdot t \cdot h(\lambda)
		\end{align*}
		and for $j=1,2$, we observe that
		\begin{align*}
			h(\delta) =
			h \left( \frac{1 - \left(\frac{\lambda_1}{\lambda}\right)^t}
			{1 - \left(\frac{\lambda_1}{\lambda_2}\right)^t} \right)
			&\leq 2 h(1) + 4 t \cdot h(\lambda) + 2 \log 2\\
			&\leq 4.181 \cdot t \cdot h(\lambda).
		\end{align*}
		Therefore, $h(\delta)$ is the same for all types $j$ and we get
		\begin{align*}
			h(\varphi) &\leq A \cdot h(\lambda) 
			+ h \left(\frac{\alpha - \alpha_2}{\alpha - \alpha_1}\right)\\
			& \leq \left( A + 4.181 t\right) h(\lambda)\\
			&\leq  5.181 \max \{A, t\} h(\lambda)\\
			&\leq \frac{5.181}{3} \max\{A, t\} \log n.
		\end{align*}
	\end{proof}

	We distinguish cases according to whether $\max\{A,t\} = A$ or 
	$\max\{A,t\} = t$.
	In both cases, we will apply Corollary~\ref{cor-laurent-2} on
	$\Lambda = \log |\varphi| + B \log |\lambda_2|$,
	so we set
	\begin{align}
		\label{eq-logA1-logA2}
		\begin{split}
			\log A_1 &= 1.73 \max\{A, t\} \log n > h(\varphi),\\
			\log A_2 &= \frac{1.01}{3} \log n > h(\lambda_2).
		\end{split}
	\end{align}
	
	\begin{lemma}
		Let $t < \frac{\log |y|}{\log n}$ and $A \geq t$.
		Then, non-trivial solutions can only exist if $n < 275539$.
	\end{lemma}
	\begin{proof}
		According to \eqref{eq-logA1-logA2}, we set
		\begin{align*}
			b' &= \frac{B}{5.181 A \log n} 
				+ \frac{1}{1.01 \log n}\\
			&< \frac{1.01 n \log n}{5.181 \log n} + 0.1\\
			&< \frac 15 n.
		\end{align*}
		
		Laurent's Corollary~\ref{cor-laurent-2} yields
		\begin{align*}
			\log |\Lambda| &> - 17.9 \cdot 3^4 
				\left(\max \left\{\log \left(\frac n5 + 0.21 \right),10 \right\}\right)^2
				\left( \frac{1.0001}{3} \right)^2 \cdot 5.181 \cdot  A (\log n)^2 \\
			&> - 834.83 \cdot A \cdot (\log n)^4.			
		\end{align*}
		From \eqref{eq-large-el-bounds-L}, we know that
		\begin{align*}
			3 \log |y|+ t \log(0.99 n) - \log(8) < - \log |\Lambda|,
		\end{align*}
		and consequently
		$$ 3 \log |y| < 835 \cdot A \cdot (\log n)^4.$$
		
		Together with \eqref{eq-large-B-bound}, we obtain
		\begin{align*}
			B &< 6.302 \frac{\log |y|}{\log n}\\
			&< \frac{6.302}{3} 835 \cdot A \cdot (\log n)^3\\
			&< 1754.06 \frac{B}{0.999 n \log n} (\log n)^3\\
			&< 1756 \frac{(\log n)^2}{n} B.
		\end{align*}
		Thus, $n < 1756 (\log n)^2$ or $n < 275539$.
	\end{proof}

	\begin{lemma}
		\label{lem-A<t}
		Let $A < t < \frac{\log|y|}{\log n}$.
		Then
		$$ \log |\Lambda| \geq - 835 \left(
			\max \left\{ \log\left(\frac{B}{t \log n}\right),
				10 \right\} \right)^2 t (\log n)^2.$$
	\end{lemma}
	
	\begin{proof}
		We note that 
		$$t < \frac{\log |y|}{\log n} < \frac{1}{2.849} B,$$
		and with \eqref{eq-logA1-logA2}, we obtain
		\begin{align*}
			b' &= \frac{B}{5.181 \cdot t \log n} + \frac{1}{1.01 \log n}\\
			&<  \frac{B + 5.13 t}{5.181\cdot t \log n}\\
			& \leq \frac{\left (1 + \frac{5.13}{2.849} \right ) B}
			{5.181 \cdot t \log n}\\
			& < 0.55 \frac{B}{t\log n}.
		\end{align*}
		Since $\log 0.55 + 0.21 < 0$, we get
		\begin{align*}
			\log |\Lambda| &\geq - 17.9 \cdot3^4 \left(
			\max \left\{ \log\left(\frac{B}{t \log n}\right), 
			10 \right\} \right)^2
			\left(\frac{1.0001}{3} \right)^2
			\cdot 5.181 \cdot t (\log n)^2\\
			&\geq - 835 \left(
			\max \left\{ \log\left(\frac{B}{t \log n}\right),
			10 \right\} \right)^2 t (\log n)^2.
		\end{align*}
	\end{proof}
	
	\begin{lemma}
		\label{lem-bounds-exist}
		Let $A < t < \frac{\log|y|}{\log n}$ and 
		$\log\left(\frac{B}{t \log n}\right) > 10$.
		Then
		\begin{align*}
			n &< 6.48 \cdot 10^{12},\\
			t &< 4.774\cdot 10^6,\\
			B &< 2.2071 \cdot 10^{14},\\
			\log |y| &< 2.63 \cdot 10^{15}.
		\end{align*}
	\end{lemma}
	\begin{proof}
		We know that
		\begin{align}
			\label{eq-bound-Lambda-B-logn}
			\begin{split}
				\log |\Lambda| &< \log \left(\frac{8}{|y|^3 n^t}\right)\\
				&< -3 \log|y| \\
				&< -{\frac{3}{6.302}}B \log n.
			\end{split}
		\end{align}
				
		Combined with Lemma~\ref{lem-L'-laurent} and \ref{lem-A<t} we get
		\begin{align*}
			B \log n &< 2.101 \cdot 835 \cdot \log\left(\frac{B}{t \log n}\right)^2 
			\cdot t \cdot
			(\log n)^2\\
			&< 1754.1 \cdot \log\left(\frac{B}{t \log n}\right)^2 
			\cdot 173 \left( \log \left( \frac{B}{\log n} \right) \right)^2
			(\log n)^3,
		\end{align*}
		or
		\begin{align}
			\label{eq-B-logn}
			\frac{B}{\log n} < 303452 \cdot \log n
			\left( \log \left( \frac{B}{\log n} \right) \right)^2
			\left( \log \left( \frac{B}{t\log n} \right) \right)^2.
		\end{align}
		
		Since $0.999n\log n < B < 1.01 n \log n$, it follows that
		\begin{align*}
			0.999 n < 303452 \log n \left(\log(1.01 n)\right)^2
				\left(\log\left(\frac{1.01 n}{2}\right)\right)^2.
		\end{align*}
		Thus, by solving the inequality, we get
		\begin{equation}
			\label{eq-bound-n}
			n < 6.48 \cdot 10^{12}.
		\end{equation}
		
		We use this bound for $n$ in \eqref{eq-B-logn} (and the lower bound 
		$10^5 \leq n$),
		which yields
		\begin{align*}
			B <  2.2071 \cdot 10^{14}.
		\end{align*}
		
		To get a bound for $t$, we use Lemma~\ref{lem-L'-laurent} and we see that
		\begin{align*}
			t &< 173 \cdot \left(\log\left(\frac{B}{\log n}\right)\right)^2 \log n\\
			& < 4.774 \cdot 10^6.
		\end{align*}
		
		From \eqref{eq-large-B-bound} we know that
		$$\log |y| < \frac{1}{2.849} B \log n,$$
		so we get
		$$\log |y| < 2.63 \cdot 10^{15}.$$
		
	\end{proof}
	
	\begin{lemma}
		\label{lem-bounds-max10}
		Let $10 > \log\left(\frac{B}{t \log n}\right).$
		Then
		\begin{align*}
			n &< 6.135\cdot 10^{11},\\
			B &< 1.76 \cdot 10^{13},\\
			t &< 3.696 \cdot 10^6,\\
			\log |y| &< 1.677 \cdot 10^{14}.
		\end{align*}
	\end{lemma}
	\begin{proof}
		Analogously to Lemma~\ref{lem-bounds-exist} we obtain
		\begin{align*}
			\frac{B}{\log n} &< 2.101 \cdot 835 \cdot 100 \cdot t \log n\\
			&< 175433.5\cdot 173 \left(\log \left(\frac{B}{\log n}\right)\right)^2 \log n.
		\end{align*}
		Hence, 
		\begin{align*}
			0.99 n < 30350000 \cdot \left(\log (1.01 n)\right)^2 \cdot \log n
		\end{align*}
		and we get
		\begin{align*}
			n < 6.135 \cdot 10^{11}.
		\end{align*}
				
		We insert this bounds for $n$ into
		\begin{align*}
			B &< 30350000 
				\left(\log \left(\frac{B}{\log n}\right)\right)^2
				(\log n)^2,
		\end{align*}
		which yields
		\begin{align*}
			B <  1.76 \cdot 10^{13} .
		\end{align*}
		
		Using the respective bounds in Lemma~\ref{lem-L'-laurent} we get
		\begin{align*}
			t &< 173 \cdot \left(\log\left(\frac{B}{\log n}\right)\right)^2 \log n\\
			&<  3.696\cdot 10^6.
		\end{align*}
		
		Finally, from
		$$\log |y| < \frac{1}{2.849} B \log n,$$
		we conclude that
		$$\log |y| < 1.677 \cdot 10^{14}.$$
	\end{proof}

    The results established in this section may be summarized as follows.
    
    \begin{prop}\label{prop-bounds-laurent}
        Non-trivial solutions for \eqref{eq-lev-wald} can occur only in the following cases:
        \begin{itemize}[leftmargin=2cm]
            \item[Case 1)] If $t\geq \frac{\log |y|}{\log \lambda}$, then $ n\leq 157\ 442.$
            \item[Case 2)] If $t < \frac{\log|y|}{\log\lambda}$, then:
            \begin{itemize}
                \item[a)] If $A\geq t,$ then $ n\leq 275\ 539$.
                \item[b)] If $A<t$, then: \vspace{-1em}
                \begin{align*}
                    n &< 6.48 \cdot 10^{12},\\
        			t &< 4.774\cdot 10^6,\\
        			B &< 2.2071 \cdot 10^{14},\\
        			\log |y| &< 2.63 \cdot 10^{15}.
                \end{align*}
            \end{itemize}
        \end{itemize}
    \end{prop}

	\subsection{Reduction by Laurent}
	\label{subsec-red-laurent}
	The largest bounds occur in Case~2b of Proposition~\ref{prop-bounds-laurent},
    so we use these	as a-priori upper bounds in Theorem~\ref{theo-laurent-2}:
	\begin{align*}
		n_0 &= 6.48\cdot 10^{12},\\
		t_0 &= 4.774\cdot 10^6,\\
		B_0 &= 2.2071 \cdot 10^{14},\\
		\log |y_0| &= 2.63 \cdot 10^{15}.
	\end{align*}
	
	Further, we set lower bounds for $n$ and $t$,
	\begin{align*}
		n' &= 1.45\cdot10^7 \\
		t' &= 2.
	\end{align*}
    The lower bound $n'$ is chosen in this way so that the repeated application of 
    Laurent's Theorem~\ref{theo-laurent-2} yields a contradiction, from which it
    follows that $n<n'$. 
    
	First, we will apply the theorem on 
	$\Lambda' = A \log \lambda + B\log\abs\lambda_2$.
	Using our a-priori bounds from equation \eqref{eq-large-B-bound} 
	and Lemma~\ref{lem-est-AB-nlogn}, we get
	\begin{align}
		\label{eq-laurent-apriori-AB}
		\begin{split}
			B &< \min \left\{B_0, 3.151 \cdot \left(
				\frac{\log |y_0|}{\log n'}+t_0 
			\right)\right\}\\
			A &< \frac{B}{0.999 \cdot n' \log(n')}.
		\end{split}
	\end{align}
	
	We set the parameters 
	$$\mu = 0.6 \ \text{ and }\ \rho = 25.$$
	Now, we apply Theorem~\ref{theo-laurent-2} to get a numerical bound on 
    $\log |\Lambda'|$.
    We obviously have $D=3$.
    Further, we can immediately compute $\sigma$ and $\lambda$.
    Next, we set $a_1$ and $a_2$ such that \textit{(ii)} of Theorem~\ref{theo-laurent-2}
    is fulfilled and $h$ such that \textit{(i)} is satisfied.
    We compute $H$, $\omega$ and $\theta$ and check whether \textit{(iii)} is fulfilled.
    Now we can compute $C$ and $C'$, and finally we get a numerical bound on 
    $\log |\Lambda'|$,
    $$-\log |\Lambda'| < 1.694 \cdot 10^7.$$
    
    Now we can use inequality \eqref{eq-large-el-bounds-L'} to compute a new bound 
    for $t$, 
	$$ t_1 = \frac{-\log |\Lambda'| + \log 10.1}{\log n'} = 1026859.$$
	
	Next, we apply Theorem~\ref{theo-laurent-2} on 
	$$\Lambda = \log \varphi + B \log |\lambda_2|.$$
	We use the estimate for $B$ given in \eqref{eq-laurent-apriori-AB}. But, instead of the a-priori bound $t_0$ we can use the improved bound $t_1$.
    
	Further, we note that $\delta < 1.0001$, thus
	$|\varphi| < 1.0001\cdot n_0 ^A $ and
	$$h(\varphi) = \frac 13 \log n_0 \cdot \left(4.08 \cdot t_1 + 1.01 A\right).$$
	
	We choose
	$$\mu = 0.6 \ \text{ and }\ \rho = 45$$
	and we apply Laurent's Theorem \ref{theo-laurent-2} the same way as before on $\Lambda'$.
    We get a new numerical bound for
	$\log|\Lambda|$,
    $$ -\log |\Lambda| < 5.7 \cdot10^{12}$$
	With the bound from \eqref{eq-bound-Lambda-B-logn}, we can then reduce
	our bound for $B$ to
	$$B_1 = -\log|\Lambda| \cdot \frac{2}{\log n'} = 7.262 \cdot 10^{11}.$$
	
	To get a bound for $n$, we insert $B_1$ into
	$ n \log n < \frac{B}{0.99}$
	and get
	$$n < n_1 = 3.013\cdot 10^{10}. $$
	
	We know from \eqref{eq-large-el-bounds-L} that
	$$\log|y| < \frac{1}{3} \left(\log |\Lambda| +\log 8 - t \log n\right) < 1.9\cdot10^{12}.$$
	
	We can reduce $t$ once more.
	From Lemma~\ref{lem-L'-laurent}, we know that
	$$t < 173 \log n_1 \left(\log \left(\frac{B_1}{\log n'}\right)\right) < 102303.$$

	Now we iterate this process to further reduce our bound.
	After 10 iterations, we obtain $A = 0$.
	Since $A=0$ implies $B=0$, we get only the trivial solution, ergo we know that
	for $n \geq 1.45 \cdot 10^7$ there are no non-trivial solutions.

	We use this bound on $n$ and the bounds on $t$, $B$ and $\log|y|$ 
    obtained in the ninth iteration and improve Case~2b of 
    Proposition~\ref{prop-bounds-laurent} as follows:
	\begin{prop}\label{prop-bounds-reduced}
        Non-trivial solutions for \eqref{eq-lev-wald} can occur only in the following cases:
        \begin{itemize}[leftmargin=2cm]
            \item[Case 1)] If $t\geq \frac{\log |y|}{\log \lambda}$, then $ n\leq 157\ 442.$
            \item[Case 2)] If $t < \frac{\log|y|}{\log\lambda}$, then:
            \begin{itemize}
                \item[a)] If $A\geq t,$ then $ n\leq 275\ 539$.
                \item[b)] If $A<t$, then: \vspace{-1em}
                \begin{align*}
                    n &< 1.45 \cdot 10^{7},\\
        			t &< 50910\\
        			B &< 4.72 \cdot 10^{9},\\
        			\log |y| &< 1.24 \cdot 10^{9}.
                \end{align*}
            \end{itemize}
        \end{itemize}
    \end{prop}
		
\section{
	Solving the remaining Thue euations}
\label{sec-small}

We have found an upper bound for the possible values for $n$ such that 
\eqref{eq-lev-wald} can possibly have non-trivial solutions.

In view of the bounds obtained in Proposition~\ref{prop-bounds-reduced}, 
we fix $2\leq n \leq 1.45 \cdot 10^7$ and proceed for fixed $n$ as follows:
\begin{enumerate}
    \item We use Laurent's Corollary~\ref{cor-laurent-2} to get an upper bound for $t$.
    Further, we derive bounds for $\log|y|$ and $|B|$ that depend only on $t$. 
    \item We reduce the bound on $t$ by iteratively applying Lemma~\ref{lem-cont-fr}.
    \item We solve Case~2. First, for $n\leq 275539$, we show that non-trivial 
    solutions can only exist if $n\leq 15$ (Lemma~\ref{lem-red-baker-dav}), 
    therefore Case~2a is (almost) done.    
    Next, we consider Case~2b, where $A<t$ and $275540\leq n\leq 1.45\cdot10^7$.
    We show by numerical computations that most combinations $(n,t,A, B)$ cannot have 
    non-trivial solutions.
    This step requires the main computational effort, taking approximately 36 hours 
    on a standard notebook (see the end of the paper for hardware specifications).
    The code is publicly available on Gitlab at
    \url{https://git.sbg.ac.at/b1065846/complete-resolution-of-a-family-of-twisted-thue-equations}.
    For the remaining values of $n$ and $t$, we solve the associated Thue equations.
    \item We solve Case~1 by solving all Thue equations with $n < 157442$ and $2\leq t \leq \bar{t}(n).$
\end{enumerate}
At this point, we have solved Equation \eqref{eq-lev-wald} for $n\geq 2$, so we need one more step for a complete solution.
\begin{enumerate}
    \setcounter{enumi}{4}
    \item We solve Equation \eqref{eq-lev-wald} for $n=0,1$.
\end{enumerate}

\subsection{Finding a (small) bound for $t$}

For a fixed $n$ in the range  $2\leq n < 1.45 \cdot 10^7$, we aim to find a small
upper bound for $t$. To achieve this, we perform the following computations using
Sage \cite{sagemath}. 
We compute the corresponding roots with 200-bit precision, thus we treat 
$\lambda,\lambda_1$ and $\lambda_2$ as actual numerical values in the following.
The constants $c_{12},\dots, c_{15}$ are computed numerically for each $n$ and 
depend only on $\lambda$.

\begin{lemma} \label{lem-bound-logy-dep-t}
We have $\log|y| < c_{12}\cdot t\log t$.
\end{lemma}

\begin{proof}
    We use the Theorem of Baker and Wüstholz~\ref{theo-baker-wust} to find a lower bound on
    	$$ \Lambda = A \log |\lambda| + B \log |\lambda_2| + \log |\delta| $$
    with
    \begin{align*}
    	\delta =
    	\begin{cases}
    		\frac{\alpha - \alpha_2}{\alpha - \alpha_1} &\text{if } j = 0,\\
    		\frac{1 - \left(\frac{\lambda_j}{\lambda}\right)^t}
    		{1 - \left(\frac{\lambda_1}{\lambda_2}\right)^t}
    		&\text{if }j = 1,2.
    	\end{cases}
    \end{align*}
    
    We know that
    $$ h(\lambda) = h(\lambda_2) = \frac 13 \log(\lambda+1)$$
    and
    $$h\left(\delta\right) \leq 4t \cdot h(\lambda) + 2 \log 2,$$
    thus we get
    \begin{align*}
    	\log |\Lambda| 
    	&> - C(3,3) \log B \cdot h(\lambda)^2 \cdot h(\delta)\\
    	&>  -c_{10} \cdot t \log B.
    \end{align*} 
    Further, we know that
    $$\log c_5 - 3 \log |y| - t \log\abs{\lambda\lambda_2^2} >\log |\Lambda|, $$
    hence we get
    \begin{align*}
    	\log|y| &< \frac 13 \left(  c_{10} + \log\abs{\lambda\lambda_2^2} + \frac{\log c_5}{2}\right) \cdot t \log B.\\
    \end{align*}
    We assume for the moment that
    \begin{equation}
    	\label{eq-max-B}
    	\frac{\log|y|}{\log\lambda} \geq t,
    \end{equation}
    thus $B < 2\cdot c_8 \frac{\log|y|}{\log\lambda} $.
    Then we have
    \begin{align*}
    	\log |y| &< \frac 13\left(c_{10} + \log\abs{\lambda\lambda_2^2} + \frac{\log c_5}{2} \right)  \cdot t \cdot \log\left(2\cdot c_8 \frac{\log|y|}{\log \lambda}\right)\\
    	&= \frac 13\left(c_{10} + \log\abs{\lambda\lambda_2^2} + \frac{\log c_5}{2} \right)  \cdot t \cdot 
    	\left(\log\left(\frac{2\cdot c_8}{\log \lambda}\right)+ \log \log|y|\right)\\
    	&= c_{11} \cdot t \log\log |y|.
    \end{align*}
    Applying Lemma~\ref{lem-petho-dew}, we then obtain
    \begin{equation}
    	\label{eq-res-matv}
    	\log |y| \leq 2\cdot c_{11} \cdot t \cdot \log \left(c_{11}\cdot t\right) = c_{12} \cdot t \log t.
    \end{equation}
        
    If \eqref{eq-max-B} is not fulfilled, we have $$\log |y| < t \log \lambda.$$
    It is easy to see that $c_{10} >  \left(\log \lambda\right)^2$, thus $c_{12} > \log|\lambda|$,
    therefore we may use \eqref{eq-res-matv} as a valid upper bound in all cases.
\end{proof}

\begin{lemma} \label{lem-bound-B-dep-t}
    We have $|B| < c_{13}\cdot t \log t$.
\end{lemma}
\begin{proof}
    Since $|B| < c_8 \left(\frac{\log|y|}{\log\lambda}+t\right),$ we get
    \begin{align*}
    	|B| < c_8 \left(c_{12} \frac{t\log t}{\log \lambda} + t\right) = c_{13} \cdot t \log t.
    \end{align*}
\end{proof}    

\begin{lemma}
    We have $t < c_{15}$.
\end{lemma}
\begin{proof}
    Now we apply Laurent's Lemma~\ref{cor-laurent-2} on $\Lambda'$. 
    We set
    $$ b' = \frac{A+B}{\log (\lambda+1)} < \frac{2 c_{13} \cdot t \log t}{\log (\lambda+1)}.$$
    Since $c_{13}$ arises from the Baker–Wüstholz Theorem, it is very large, 
    consequently $\log b' + 0.38 >10.$
    It follows that
    \begin{align*}
    	\log |\Lambda'| &> - 17.9 \cdot 3^4  \left( \max \left\{\log b'+0.38, 10\right\}\right)^2 h(\lambda)^2\\
    	&> - 17.9 \cdot 3^4
    	\left( \log\left(\frac{2  c_{13} \cdot t \log t}{\log(\lambda+1)}\right) + \log(1.5) \right)^2
    	\left( \frac 13 \log(\lambda+1) \right)^2\\
    	&= - 17.9 \cdot 9 
    	\left( \log t +\log \log t 
    	+ \log \frac{2\cdot 1.5\cdot c_{13}}{\log (\lambda+1)}\right)^2 \log(\lambda+1)^2\\
    	&= - c_{14} \cdot \log t.
    \end{align*}
    
    We know from Lemma~\ref{lem-bound-L'} that $|\Lambda'| < c_9 |\lambda\lambda_2^2|^{-t}$, 
    thus we get
    $$ t \log \abs{\lambda \lambda_2^2} \leq c_{14} \log t + \log(c_9),$$
    or equivalently
    $$ t < c_{15}.$$
\end{proof}

Substituting this bound into Lemmas~\ref{lem-bound-logy-dep-t} and \ref{lem-bound-B-dep-t}, we get
the absolute bounds
\begin{align*}
    \log |y| &< c_{12} \cdot c_{15} \log c_{15},\\
    |A| < |B| &< c_{13} \cdot c_{15} \log c_{15}.
\end{align*}

\subsection{Reducing $t$} 
To reduce the bound for $t$, we apply the following lemma from \cite[p47]{baker-concise-introduction}.

    \begin{lemma}
		\label{lem-cont-fr}
		Let $\mu \in \RR\backslash\QQ$, $\mu = [a_0; a_1,a_2,\dots]$.
		Let $l\in \ZZ^+$, set $\tilde{A}=\max_{i = 1, \dots, l}\{a_i\}$ and let
		$\frac{p_l}{q_l}$ be the $l$-th convergent to $\mu$.
		Then 
		$$\frac{1}{(2+\tilde{A})q_l} < \abs{p-q\mu}$$
		for any rational fraction $\frac pq$ with $q\leq q_l$.
	\end{lemma}

We set $\mu = \frac{\log\abs{\lambda_2}}{\log\lambda}$ and compute the
corresponding continued fraction $[a_0;a_1,a_2,\dots]$.
We then compute the denominators $q_l$ of the convergents by
$$q_0 = 1, \quad q_1 = a_1, \quad q_l = a_l q_{l-1} + q_{l-2}$$
At some point, we reach a denominator $q_l$ that exceeds our bound for $B$.
At this point, we may apply Lemma~\ref{lem-cont-fr} with $\frac pq = \frac AB$
and $\tilde{A} = \max_{i=1,\dots,l}\{a_i\}$,
which yields
	\begin{align*}
		\frac{1}{(2+\tilde{A})q_l} &< \abs{A -B \frac{\log\abs{\lambda_2}}{\log\lambda}}\\
		&\leq \frac{c_9}{\log \lambda} \abs{\lambda\lambda_2^2}^{-t}.
	\end{align*}
Applying the logarithm, we obtain a new bound for $t$,
	\begin{align*}
		t \log\abs{\lambda\lambda_2^2} 
			\leq \log \left ( \frac{c_9}{\log\lambda} \right ) 
				+ \log\left (2+\tilde{A}\right )+\log q_l,
	\end{align*}
which is significantly smaller than the previous one.

Using the new bound in Lemma~\ref{lem-bound-B-dep-t}, we get a sharper upper
bound for $|B|$, allowing us to apply Lemma~\ref{lem-cont-fr} again with the 
refined bounds.

We give a few examples of the refined bounds after 3 iterations:
\begin{center}
	\begin{tabular}{ c| ccccccc} 
		$n$ & 2& 10 & 100 & 1000 & $10^5$ & $10^6$ & $10^7$ \\
		\hline
		$\bar{t}(n)$ & 83 &22 & 11 & 8 & 5 & 4 & 3
	\end{tabular}
\end{center}

Since the bound on $\log|y|$ in Lemma~\ref{lem-bound-logy-dep-t} 
depends only on $t$, we also get an improved bound on $\log |y|$.

\subsection{Case 2a}

We start by looking at  Case~2, where $t < \frac{\log|y|}{\log\lambda}$.

The bound for Case~2a is $n < 275539$.
We consider these values for $n$ in both Cases~a and b.

To solve this case, we use the Baker-Davenport reduction \cite{baker-davenport} given in the form of Odjoumani and the last author in \cite{odjoumani-ziegler}.
	
\begin{theorem}[Baker-Davenport]
    \label{theo-baker-dav}
    Consider a Diophantine inequality of the form
    $$ \abs{ n \mu + \tau - x } < c_1 \exp\left( -c_2 n \right),$$
    where $n \in \NN$, $x \in \ZZ$, $c_1$ and $c_2$ are positive constants, 
    and $\mu$ and $\tau$ are real numbers.
    Suppose $n < N$ and $\kappa > 1$ such that there exists a convergent
     $p/q$ to $\mu$, where
    $$\{q\,\mu\} < \frac{1}{2\kappa N} \qquad \text{and} \qquad
        \{q\,\tau\} > \frac{1}{\kappa}, $$
    where $\{\cdot\}$ denotes the distance to the nearest integer.
    Then we have
    $$n \leq \frac{ \log\left(2\kappa q c_1\right) }{c_2}.$$
\end{theorem}

An application of Theorem~\ref{theo-baker-dav} will yield the following result:
\begin{lemma}
	\label{lem-red-baker-dav}
	If $t < \frac{\log |y|}{\log \lambda}$ and $n \leq 275539$, then non-trivial solutions 
    can only exist for
	$$n \leq 15.$$
\end{lemma}

\begin{proof}
	We know from Lemma~\ref{lem-bound-logy-dep-t} that $\log |y| < 2 c_{11} \cdot t \cdot \log(c_{11}t)$.
	This bound is very large, as it contains a Baker bound.
	We use it as an a-priori bound for the following.
	
	We apply Theorem~\ref{theo-baker-dav} on
	\begin{equation*}
		\abs{\frac{\Lambda}{\log \lambda}} = 
			 A + B \frac{\log |\lambda_2|}{\log \lambda}
			 + \frac{\log \abs{\frac{\alpha_j - \alpha_l}{\alpha_j - \alpha_k}}}
			 	{\log \lambda} 
			\eqqcolon A + B \mu + \tau(j)
	\end{equation*}
	to reduce $|B|$ and subsequently $|y|$.
	
	We have the upper bound
	\begin{equation*}
		\abs{\frac{\Lambda}{\log \lambda}} < 
		\frac{c_5}{|\lambda\lambda_2^2|^t \log \lambda} \cdot \frac{1}{|y|^3}.
	\end{equation*}
	Further, we know from \eqref{eq-bound_B} that
	\begin{align*}
		|B| &< c_8 \left(\frac{\log |y|}{\log \lambda} + t\right) < \log |y| \cdot \frac{2 c_8}{\log \lambda},
	\end{align*}
	which leads to
	\begin{align*}
		|y| \geq \exp \left(\frac{|B|\cdot \log \lambda}{2 c_8}\right).
	\end{align*}
	
	Therefore,
	\begin{align*}
		\abs{A + B\mu + \tau(j)} < \frac{c_5}{\abs{\lambda\lambda_2^2}^t \log \lambda} 
			\cdot \exp \left(-3 \frac{ \log \lambda}{2c_8} \cdot |B|\right).
	\end{align*}
	
	Now we apply Theorem~\ref{theo-baker-dav}. 
    We compute the first convergent $p/q$ of $\mu$, and we set
	\begin{align*}
		\kappa = \frac{1}{2 c_8 \left(\frac{\log |y|}{\log \lambda} +t\right) \{q\mu\}},
	\end{align*}
	i.e. $\{q\mu\} < \frac{1}{2 \kappa N}$.
    We check whether
    $\kappa > 1$ and $\kappa\{q\tau\} > 1$. If this is not fulfilled, we take the next 
    convergent of $\mu$ and so on,
    until the two conditions are satisfied.
    If we do not find such convergents within our precision, we cannot reduce $|B|$ and 
    consequently $\log|y|$.

    We get a reduced bound
	\begin{align}
    \label{eq-red-bound-B}
		|B| \leq \log \left(2 \kappa q \frac{c_5}{|\lambda\lambda_2^2|^t\log \lambda} \right)
			\frac{2c_8}{3 \log\lambda},
	\end{align}
	which is rather small.
	But we know from Lemma~\ref{lem-est-AB-nlogn} that $B > 0.99\cdot n\log n$.
	We check this inequality and see that it is not fulfilled for $n\geq 16$,
	thus all those $n$ are eliminated and we get
	$$n\leq 15.$$
\end{proof}

To solve the Thue equations where $n\leq15$, we use the subsequent lemma.

\begin{lemma}
	\label{lem-sol-conv}
	Let $(x,y)$ be a solution of \eqref{eq-lev-wald} of type $j$.
	Then $\frac{x}{y}$ is a convergent of the continued fraction of $\lambda_j^t$
\end{lemma}

\begin{proof}
	From Lemma~\ref{lem-bound-bj} we know that
	$$ \abs{\beta_j} = \abs{x - \lambda_j^t \cdot y} 
	< \frac{c_1}{|y|^2} \cdot 
	\begin{cases}
		\abs{\lambda}^{-2t} &\text{if }j=0,\\
		\abs{\lambda \lambda_2}^{-t} &\text{if }j=1,2.
	\end{cases}	$$
	
	We divide the inequality by $|y| \geq 2$ and get
	\begin{align*}
		\abs{\frac xy - \lambda_j^t} < \frac{c_1}{|y|^3} \cdot 
		\begin{cases}
			\abs{\lambda}^{-2t} &\text{if }j=0,\\
			\abs{\lambda \lambda_2}^{-t} &\text{if }j=1,2.
		\end{cases}
	\end{align*}
	
	For ease of notation, let us only look at the (worse) case $j=1,2$, as then the case $j=0$ is fulfilled anyway.
	If $c_1 |\lambda\lambda_2|^{-t} < \frac 12,$ we can apply 
	Lemma~\ref{theo-convergent}.
	This holds for all $n\geq 2$, i.e. for all $n$ in question -- we need to treat $n=0,1$ separately anyway.	
	Since we have confirmed that
	$$ \abs{\lambda_j^t - \frac xy} < \frac{1}{2 |y|^2},$$
	it follows from Theorem~\ref{theo-convergent} that $\frac xy$ is a convergent
	of the continued fraction of $\lambda_j^t$.
\end{proof}

Therefore, we compute the convergents of $\lambda_j^t$ until their denominators
exceed the upper bound for $|y|$ established below.

For $n\leq 15$, using the reduced bound for $|B|$ from \eqref{eq-red-bound-B}, we use \eqref{eq-bound_B} 
to get a reduced bound for 
$\log|y|$ by computing
$$ \log |y| < \frac{|B| \log \lambda}{c_8'}.$$
According to Lemma~\ref{lem-sol-conv}, we check the convergents of $\lambda_j^t$
up to this bound.
This is fast, and we get the solutions that
are stated in Theorem~\ref{theo-results}, 
$$ (n,t,x,y) = (2,2,\pm7,\pm1),\ (2,2,\pm2,\pm1)\ (4,2,\pm3,\pm2).$$
Hence, Case~2a is completely resolved.

\subsection{Case 2b}
From Lemma~\ref{lem-bounds-exist} and Section~\ref{subsec-red-laurent},
we know that we need to check $275739 < n < 1.45 \cdot 10^7$ for Case~2b, when 
$A<t<\frac{\log|y|}{\log n}$.
We therefore need a strategy to 
eliminate a substantial number of the possible pairs $(n,t)$.

For each $275540\leq n \leq 1.4\cdot10^7$, we look at 
	$$\abs{\frac{\Lambda'}{\log|\lambda_2|}} 
		= \abs{A \frac{\log \lambda}{\log |\lambda_2|} + B}
		< \frac{c_9}{\abs{\lambda\lambda_2^2}^t \cdot \log|\lambda_2|}
		\eqqcolon c_{16}$$
and we note that $c_{16}$ is very small, for example, when $n=100$,
we already have $c_{16} < 0.02$.
Since $B$ is an integer, it suffices to consider combinations $(n,A)$ for which
$A \frac{\log \lambda}{\log |\lambda_2|}$ is very close to an integer. 
More precisely, we compute
\begin{align*}
	\left\{ A \frac{\log \lambda}{\log |\lambda_2|} \right\}
	\qquad \text {and} \qquad
	1 - \left\{ A \frac{\log \lambda}{\log |\lambda_2|} \right\}
\end{align*}
and compare each result to $c_{16}$.
We only need to consider the pair $(n,t)$ if one of these two quantities is 
smaller than $c_{16}$ and we know that 
$B=~\left\lfloor A \frac{\log \lambda}{\log |\lambda_2|}\right\rfloor,$
respectively,
$B=~\left\lceil A \frac{\log \lambda}{\log |\lambda_2|}\right\rceil$.

We get a set of possible tuples $S_0 =\{(n,\overline t,A,B)\}$ with 206 
possible values of $n$, where $\overline{t}$ is an upper bound for the possible values of $t$.

Since we need to execute these computations for all 
$n \in \{275540,\dots,14500000\},$ this takes a long time.
If we parallelize it, using 8 parallel threads, this takes roughly 36 hours.

We see that we have the following:
\begin{align*}
	j=0:& \quad a = \frac{2A-B}{3}, \quad b = \frac{2B-A}{3},\\
	j=1:& \quad a = \frac{2A-B}{3} - t, \quad b = \frac{2B-A}{3}-t,\\
	j=2:& \quad a = \frac{2A-B}{3} + t, \quad b = \frac{2B-A}{3}+t,
\end{align*} 
thus we can discard the triples where $ \frac{2A-B}{3} \notin \ZZ$.
We are left with a reduced set $S_1 =\{(n,\overline t,A,B)\}$ of possible values.
This set contains 61 different values of $n$ and $\overline t=3$ for all $n$.

What remains to be done is checking the Thue equations with $(n,\overline{t},A,B) \in S_1$.
Since we are looking at Case~2b, we have $A<t<\frac{\log|y|}{\log \lambda}$.

For each pair $(n,t)$, $2\leq t \leq\overline{t}$, we need to solve the corresponding Thue equation.
Since we have $\overline t = 3$ for all $n\in S$, we need to solve 122 Thue equations.
The values of $|B|$ in our set $S_1$ are already suffiently small, so we compute 
a bound for $\log|y|$, using that
$$\log|y| < \frac{|B|\log\lambda}{c_8'}.$$
According to Lemma~\ref{lem-sol-conv} we check the convergents of $\lambda_j^t$
up to this bound.
We get no additional solutions.

\subsection{Case 1}

Now we look at the case
$t > \frac{\log |y|}{\log\lambda}$.
From Lemma~\ref{lem-t>logylogn}, we know that in this case we have  $n \leq 157442.$

We solve the Thue equations for $2 \leq t < \bar{t}(n)$ using Lemma~\ref{lem-sol-conv}.
We obviously have $|y| < \lambda^t$, which is a small enough bound for the
denominator of the convergents $\frac xy$ of $\lambda_j$.

Looping through $2\leq n \leq 157442$ takes less than 3 hours.
We do not get any new solutions.

\subsection{The cases $n=1$ and $n=0$}\label{sec-n01}

For $n=1$ and $t\geq3$ we can apply the same strategy as above and we get 
$t \leq 294.$
However, since the initial bound for $\log|y|$ is already sufficiently 
small, we don't need to apply the reduction from Lemma~\ref{lem-red-baker-dav}. 
Instead, we can check the convergents of $\lambda_j^t$ up to this initial bound.
We get no additional solutions.

For $n=1$ and $t=2,$ we solve the Thue equation using PARI.
We get the non-trivial solutions
$$ (x,y) = (\pm 2, \pm 1), (\pm 3, \pm 1), (\pm 7, \pm 2), (\pm 7, \pm 3).$$

For $n=0$, we sort the roots of $F(X)$ such that 
$|\lambda| > |\lambda_2| >|\lambda_1|.$
This is equivalent to saying that $\lambda<\lambda_1<\lambda_2$.

In contrast to the previous case, where $\lambda>0$, note that we now have
$\lambda<0$.
Accordingly, each occurrence of $\log\lambda$ in the above arguments must be 
replaced by $\log|\lambda|.$
Further, we note that 
$$h_m(\lambda) = \max \left\{h(\lambda), \frac{\log\abs{\lambda}}{3}, 
	\frac13\right\} = \frac{\log\abs{\lambda}}{3}. $$

For $t\geq 8$, we can apply the above method.
We get $t < 150$. Reducing the bound for $\log|y|$ and checking the convergents
of $\lambda_j^t$ according to Lemma~\ref{lem-sol-conv} yields no additional 
solutions.

For $t=2,\dots,8$ we solve the corresponding Thue equations with PARI.
Additional solutions occur for $t=2,3,5$.
\begin{itemize}
    \item For $t=2$ we get
    $$(\pm x,\pm y)=(1,1),(1,5),(2,1),(3,1),(3,2),(13,4),(14,9).$$
    \item For $t=3$, the additional solutions are
    $$ ( \pm x, \pm y ) = (2,1).$$
    \item For $t=5$, we obtain
    $$ ( \pm x, \pm y ) =(3,1), (19,-1).$$
\end{itemize}

All computations were performed on a workstation equipped with an 
Intel Core i7-1185G7 CPU (4 cores, 8 threads) and 16 GB RAM, 
running Windows 11.
All reported runtimes were obtained on this hardware.
In total, the computations required less than 39 hours.
The code used for the computations in this paper is publicly available on 
GitLab at 
\url{https://git.sbg.ac.at/b1065846/complete-resolution-of-a-family-of-twisted-thue-equations}.



	\bibliographystyle{abbrv}
	\bibliography{bibtex.bib}
	
\end{document}